\documentclass[reqno]{amsart}
\usepackage[utf8]{inputenc}
\usepackage[T1]{fontenc}
\usepackage{lmodern}
\usepackage{hyperref}
\usepackage{amsmath}
\usepackage{amsfonts}
\usepackage{amssymb}
\usepackage{cancel}
\usepackage{color}
\usepackage{lipsum}

\numberwithin{equation}{section}
\newtheorem{theorem}{Theorem}[section]
\newtheorem{proposition}[theorem]{Proposition}

\newtheorem{lemma}{Lemma}[section]

\newtheorem{cor}{Corollary}[section]

\usepackage{graphicx}
\graphicspath{ {images/} }
\usepackage{color}
\definecolor{lightgray}{gray}{0.5}

\begin{document}
\title[Gierer-Meinhardt system with Robin-Neumann boundary conditions]{GLOBAL 
ANALYSIS OF THE SHADOW GIERER-MEINHARDT SYSTEM WITH GENERAL LINEAR BOUNDARY CONDITIONS IN A RANDOM ENVIRONMENT}
\author[Antwi-Fordjour, Kim, Nkashama]{}
\maketitle

\centerline{\scshape Kwadwo Antwi-Fordjour$^1$, Seonguk Kim$^2$, and Marius Nkashama$^3$}
\medskip
{
\footnotesize
\centerline{1) Department of Mathematics and Computer Science,}
 \centerline{Samford University,}
   \centerline{ Birmingham, Alabama 35229, USA.}
 }
 \medskip
{\footnotesize
\centerline{2) Department of Mathematics,}
 \centerline{DePauw University,}
   \centerline{Greencastle, Indiana 46135, USA.}
   }
\medskip
{\footnotesize
 \centerline{3) Department of Mathematics,}
 \centerline{University of Alabama at Birmingham,}
   \centerline{Birmingham, Alabama 35294, USA.}
   } 

\begin{abstract}
The global analysis of the shadow Gierer-Meinhardt system with multiplicative white noise and general linear boundary conditions is investigated in this paper.
For this reaction-diffusion system, we employ a fixed point argument to prove local existence and uniqueness.
Our results on global existence are based on \emph{a priori} estimates of solutions.
\end{abstract}

\section{Introduction}
In 1744, Trembley's discovery in developmental biology pointed out that fragments of the small, fresh water animal called hydra can regenerate into a complete animal \cite{Trembley}. Based on Turing's (1952) idea of ``diffusion-driven instability'' \cite{Turing}, Gierer and Meinhardt \cite{Gierer} in 1972 proposed a theory of biological pattern formation that placed special emphasis on certain striking features on developmental biology, in particular, they proposed a system to model the head formation in the hydra.
Mathematical modeling of biological spatial pattern formation
has become one of the most popular areas of investigation in applied mathematics in recent times.
Many models involved in these
biological phenomena are of the general reaction-diffusion type considered in
\cite{Turing, Takashi}. Several researchers have been able to provide great insights into the underlying mechanisms of biological processes
realized by the
Gierer-Meinhardt system of the following form.

\begin{equation}\label{original equation}
  \begin{cases}
 \vspace{3pt}
A_{t} =\epsilon^{2}\Delta A-A+\dfrac{A^{p}}{H^{q}+b} \quad &\text{in} \: D\times (0,T), \\
\vspace{3pt}
\tau H_{t} =d\Delta H-H+\dfrac{A^{r}}{H^{s}} \quad &\text{in} \: D\times (0,T),\\
\epsilon\dfrac{\partial A}{\partial\nu}+aA=0=\dfrac{\partial H}{\partial \nu} \quad &\text{on} \:\partial D\times (0,T),\\
H(x,0)=H_{0}(x)>0, \quad A(x,0)=A_{0}(x) \geq 0 \quad &\text{in} \: \overline{D},
\end{cases}
  \end{equation}
where $\epsilon>0,~d>0, ~\tau>0 , ~a\geq 0$, $b>0$ and $D \subset \mathbb{R}^{N}$ $(N\geq 1)$ is a bounded domain with a smooth boundary $\partial D$, and $A$ and $H$ are activator and inhibitor, respectively; $\Delta$ is the Laplace or diffusion operator in $\mathbb{R}^{N}$ acting on $A$ and $H$; $\nu(x)$ is the unit outer normal vector at $x \in \partial D $, $\partial/\partial\nu:=\nabla\cdot\nu$ is the directional derivative in the direction of the vector $\nu$. The reaction exponents $p,q,r,$ and $s$ are positive, and satisfy $(p-1)(s+1)<qr$. The constants $\epsilon$ and $d$ are the diffusion coefficients for the activator and inhibitor respectively. The constant $b$ provides additional support to the inhibitor and may be thought of as a measure of the effectiveness of the inhibitor  in suppressing the production of the activator. The time relaxation constant $\tau$ plays a significant role on the stability of the system.  The two chemical substances $A$ and $H$, representing the concentrations of certain biochemicals, are initially produced by an outside source. Then they interact as represented by the coupled nonlinear terms in the system (see e.g. \cite{Masuda} and references therein).

There are several results for equation \eqref{original equation} with homogeneous linear Neumann boundary conditions (i.e.,
$a = 0$ and $b=0$) in \cite{Rothe, Masuda,  Dillon, Henghui, Jiang} and references therein. Chen et al. \cite{Chen} studied the generalized (singular) Gierer-Meinhardt system with Dirichlet boundary conditions. Recently, Antwi-Fordjour and Nkashama \cite{Antwi} studied the  global existence of \eqref{original equation}.
It is well known that it  is quite challenging to study the solvability of the equation \eqref{original equation} since it
does not have a standard variational structure.

One way to initiate the study of \eqref{original equation} is to first examine the shadow
system suggested by Keener \cite{Keener}. Shadow systems are mostly employed to approximate the reaction-diffusion
systems when one of the diffusion coefficients is large. Indeed, when the diffusion coefficient  of the second equation in \eqref{original equation}  is sufficiently large; that is,
$d \rightarrow \infty$, and $\gamma(t)$ is the formal limit of $H(x,t)$, then the system \eqref{original equation} can be reduced to the shadow Gierer-Meinhardt system:
\begin{equation}\label{original shadow equation}
 \begin{cases}
A_{t} =\epsilon^{2}\Delta A-A+\dfrac{A^{p}}{\gamma^{q}+b} \quad &\text{in} \: D\times (0,T), \\
\tau \gamma^{\prime} =-\gamma+\dfrac{\overline{A^{r}}}{\gamma^{s}} \quad &\text{in} \:  (0,T),\\
\epsilon\dfrac{\partial A}{\partial\nu}+aA=0 \quad &\text{on} \:\partial D\times (0,T),\\
\gamma(0)=\gamma_{0}>0  \: \text{in} ~~  \mathbb{R}, \qquad A(x,0)=A_{0}(x) \geq 0 \quad &\text{in} \: \overline{D},
  \end{cases}
  \end{equation}
where we define
$$\gamma(t):=\frac{1}{|D|}\int_{\Omega}H(x,t)dx,\qquad \overline{A^{r}}(t):=\frac{1}{|D|}\int_{\Omega}A^{r}(x,t)dx,$$
and $|D|$ is the (Lebesgue) measure of $D$. It is important to note here that, the second equation is a nonlocal  ordinary differential equation.

Global existence and finite-time blow-up for equation \eqref{original shadow equation} have been investigated by Li and Ni \cite{Li} when $a=b=0$, provided we have $\frac{p-1}{r}<\frac{2}{N+2}$. Phan \cite{Phan} studied the global existence of solutions for $a=b=0$ in  \eqref{original shadow equation} provided $\frac{p-1}{r}=\frac{2}{N+1}$ . Maini et al. \cite{Maini} studied the stability of spikes for \eqref{original shadow equation} with $b=0$.

Physical and biological systems are inevitably affected by random fluctuations from the environment. It is therefore important to incorporate the random effects from the environment into \eqref{original shadow equation}. In stochastic modeling, these random effects are conceived as stochastic fluctuations.

Motivated by the work of Kelkel and Surulescu \cite{Kelkel} and Winter et al.~\cite{Matthias}, we consider the following stochastic shadow Gierer-Meinhardt system:
\begin{equation}\label{main equation}
 \begin{cases}
 A_t =\epsilon^{2}\Delta A-A+\dfrac{A^{p}}{\gamma^{q}+b} \quad &\text{in} \: D\times (0,T), \\
\tau d\gamma =-\gamma dt+\dfrac{\overline{A^{r}}}{\gamma^{s}}dt+\sqrt{\eta}\gamma dB_{t} \quad &\text{in} \:  (0,T),\\
\epsilon\dfrac{\partial A}{\partial\nu}+aA=0 \quad &\text{on} \:\partial D\times (0,T),\\
\gamma(0)=\gamma_{0}>0  \: \text{in} ~~  \mathbb{R} , \qquad A(x,0)=A_{0}(x) \geq 0 \quad &\text{in} \: \overline{D},
  \end{cases}
  \end{equation}
where  $\eta>0$ is small and represents the noise intensity, and $B_{t}$ is a white noise (or statistically Brownian motion at time $t$).

Analytical results for the equation \eqref{main equation}  were obtained  with   Neumann boundary conditions ($a=0$) but there is a lack of theoretical considerations for the problem with general linear boundary conditions (see e.g. \cite{Matthias,  Li 1}). Thus, investigating the equation \eqref{main equation} with general linear boundary conditions of Robin-Neumann type plays an important role in understanding various kinds of biological phenomena.

\paragraph{To the best of our knowledge, this appears to be the first paper on stochastic shadow Gierer-Meinhardt system with general linear boundary conditions of Robin-Neumann type. In this paper, motivated by \cite{Matthias} and the above considerations, we shall prove the following main result on the global existence of strong positive solutions for the problem with general linear boundary conditions of Robin-Neumann type.}

\begin{theorem}\label{theorem1}
Suppose that $D\subset\mathbb{R}^N$ is a  bounded domain with a smooth boundary $\partial D$, and assume that the exponents satisfy the inequality
$$\frac{p-1}{r}<\min\Big\{\frac{2}{N+2}, \frac{q}{s+1}\Big\}.$$ Let $A_{0} \in W^{2,l}(D)$ where $l>\max\{N,2\}$, and $\gamma_0\in\mathbb{R}$ with $\gamma_0>0$. Then, with probability $1$, there is a unique solution $(A(x,t),\gamma(t))$ of the stochastic shadow equation \eqref{main equation}  which exists globally. Moreover, the component $\gamma$ satisfies the estimate
\begin{equation}\label{estimate for gamma}
\gamma(t)\geq \Big(\frac{\eta}{\tau}\Big)^{\frac{1}{s+1}} e^{-\frac{3}{2\eta}t-\frac{1}{\sqrt{\eta}} |B_{t}|}\gamma_{0}.
\end{equation}
\end{theorem}

The paper is organized as follows. In Section 2 we show the unique local existence of solutions. In Section 3 we prove global existence of positive solutions.

\section{Unique Local Existence }
In this section, we use several concepts from probability theory and semigroup of linear operators theory
 (see e.g. \cite{Billingsley, Durrett, Resnick, Henry, Pazy, Friedman}) along with estimates obtained herein and a fixed point argument to prove the unique local existence of positive solutions.
 Let us consider the (standard) probability space $(\Omega, \mathbb{F}, \mathbb{P} )$ where $\Omega$ is the sample space, $\mathbb{F}$ is the $\sigma$-algebra, $\mathbb{P}$ is the probability measure and define
\begin{equation}\label{brownian bound}
B^{*}_{t}:=\sup_{0\leq s\leq t}|B_{s}|,~~\forall t>0,\;\text{and}\; \tau_{K}(\omega):=\inf\{t>0:|B_{t}(\omega)|\geq K\}, ~\omega\in\Omega.
\end{equation}
Note that $\tau_{K}(\omega)$ denotes an optional stopping time (see e.g. \cite{Durrett} for background).
It is easy to see that
\begin{align}\label{set E}
E^{c}:=\{\omega\in\Omega:\tau_{K}(\omega)\leq t\}=\{\omega\in\Omega: B^{*}_{t}(\omega)\geq K\}.
\end{align}
Since the distribution of $B^{*}_{t}$ is a normal distribution function, we can ascertain that for sufficiently large $K>0$, we have that
$$\mathbb{P}(E^{c})<\frac{C}{K^{2}}\ll1,~~C>0;$$
which means that we can think of the complement $E^{c}$ as a negligible set.
Next,  we define the following operators;
\begin{equation}
S(t):=e^{-(-\epsilon^2\Delta+I )t} \;\text{and}\; R(t,B_t):=e^{-\frac{3}{2\eta}t+\frac{1}{\sqrt{\eta}}B_{t}}.
\end{equation}
Notice that here $S(t)$ denotes the semigroup associated with the Laplace operator subject to homogeneous Robin-Neumann boundary conditions where $(-\epsilon^2\Delta+I)$ is a strongly elliptic operator.

Consider the function space
$$C(\overline{D}, \mathbb{R})=\{f:\overline{D}\to \mathbb{R}|~ \mbox{$f$ is  a continuous function}\}$$
endowed with the sup-norm
\begin{equation}\label{def of continuous norm 1}
\|f\|_{C}=\sup_{x\in \overline{D} }|f(x)|.
\end{equation}
It follows that
\begin{equation}\label{estimate for the norm C}
\|S(t)f\|_{C}\leq \|f\|_{C},\quad \||f|^{p}\|_{C}\leq \|f\|_{C}^{p},~~p\geq 1,~f \in C(\overline{D}, \mathbb{R}).
\end{equation}
We also consider the following operator norm (on the appropriate space):
\begin{equation}
\|(A,\gamma)\|_{C([0,T];C\times  \mathbb{R})}:=\|A\|_{C([0,T];C)}+ \|\gamma\|_{C([0,T];  \mathbb{R})}.
\end{equation}
Finally, we define
$$x\wedge y :=\min\{x,y\}\;\text{and}\; x \vee y:= \max\{x,y\}\quad\text{for}\; x,y \in \mathbb{R}.$$
Based on the aforementioned preliminaries, we shall prove the following result on local existence and uniqueness of solutions to equation \eqref{main equation}.

\begin{proposition}\label{local existence}
For every $K>0$ there exists $T=T(K)>0$ such that for all $\omega$ in $E\subset \Omega$ as defined in \eqref{set E}, equation \eqref{main equation} has a unique solution $(A,\gamma) \in C([0,T \wedge \tau_K ];C(D, \mathbb{R})\times \mathbb{R})$ such that for all $t \in [0, T\wedge \tau_K]$, with $\overline\gamma(t)$ defined by $\overline{\gamma}(t):=\frac{\tau}{\eta}\gamma(t)$,
\begin{align}\label{operator setting 1}
&A(t)=S(t)A_{0}+\int_{0}^{t}S(t-u)\Big(\frac{A^{p}(u)}{\gamma^{q}(u)+b}\Big)du,
\end{align}
\begin{align}\label{operator setting 2}
&\overline{\gamma}(t)=R(t,B_t)\gamma_{0}+\int_{0}^{\frac{t}{\eta}}R(t-u,B_t-B_u) \Big(\frac{\overline{A^{r}}(u)}{\gamma^{s}(u)}\Big)du.
\end{align}
\end{proposition}

\begin{proof}
Without loss of generality, we assume in what follows that the constants $\tau=\eta=1$; which implies that $\overline\gamma(t)$ reads $\gamma(t)$. For every $\omega \in E \subset \Omega$, we first define the space
\begin{align}
\nonumber \mathbb{D}(T,K,L,\omega):=\Big\{&(A(\omega),\gamma(\omega)\in C\big([0,T\wedge \tau_K(\omega)], C(D, \mathbb{R})\times (0,\infty)\big):\\
\nonumber			&
\gamma(\omega, t)\geq e^{-\frac{3}{2}-K}\gamma_{0}, ~A(0)=A_{0},~ \gamma(0)=\gamma_{0},\\
			&\|(A,\gamma)(\omega)\|_{C([0,T\wedge \tau_K(\omega), C\times (0,\infty)])}\leq L\Big\},
\end{align}
where $T\in (0,1]$ depends on $K,L,A_0, \gamma_{0}$ with
\begin{align}\label{constant M}
L>2+\|A_{0}\|_C+e^{K}\gamma_{0}.
\end{align}
We simply denote $\mathbb{D}(T,K,L,\omega)$ by $\mathbb{D}$ (and drop all $\omega$). Next, we define the distance between $(A_{1},\gamma_{1}), (A_{2},\gamma_{2}) \in C\big([0,T\wedge \tau_K(\omega)], C(D, \mathbb{R})\times  \mathbb{R}\big)$ by
\begin{align}
d\Big((A_{1},\gamma_{1}), (A_{2},\gamma_{2})\Big):=\big\|(A_{1}-A_{2},\gamma_{1}-\gamma_{2})\big\|_{C\big([0,T\wedge \tau_K(\omega)], C\times  \mathbb{R}\big)}.
\end{align}
It is clear that the $\mathbb{D}$ is a closed metric space with the metric $d$; that is, $\mathbb{D}$ is a complete metric space. Now, consider
\begin{align}\label{operator set up 1}
& F_{1}(A,\gamma)(t):=S(t)A_{0}+\int_{0}^{t}S(t-u)\Big(\frac{A^{p}(u)}{\gamma^{q}(u)+b}\Big)du,
\end{align}
\begin{align}\label{operator set up 2}
&F_{2}(A,\gamma)(t):=R(t,B_t)\gamma_{0}+\int_{0}^{t}R(t-u,B_t-B_u) \Big(\frac{\overline{A^{r}}(u)}{\gamma^{s}(u)}\Big)du,
\end{align}
and
\begin{align}
F(A,\gamma)(t):=\big(F_{1}(A,\gamma)(t), F_{2}(A,\gamma)(t)\big).
\end{align}
In order to use the Banach fixed point theorem (i.e., the contraction mapping theorem) which guarantees the existence of a local unique pair of  solutions (i.e., a fixed-point) to \eqref{operator setting 1} and \eqref{operator setting 2}, we shall prove the following:
\begin{enumerate}
\item There exists $T:=T(K,L,\|A_{0}\|_{C},\gamma_{0})>0 $ such that
\begin{align}
F(A,\gamma)(t)\in \mathbb{D} \; \text{whenever}\; (A,\gamma)\in \mathbb{D}.
\end{align}
\item There exists $T:=T(K,L,\|A_{0}\|_{C},\gamma_{0})>0 $ such that
\begin{align}\label{distance of F}
d\Big(F(A_{1},\gamma_{1}),F(A_{2},\gamma_{2})\Big)\leq \frac{1}{2} d\Big((A_{1},\gamma_{1}),(A_{2},\gamma_{2})\Big),
(A_{i},\gamma_{i})\in \mathbb{D},~i=1,2.
\end{align}
\end{enumerate}
We first show (1). It is clear that
\begin{align}
F(A,\gamma)(0)=(A_{0},\gamma_{0}).
\end{align}
Now, let $(A,\gamma)\in \mathbb{D}$ be given. By using \eqref{def of continuous norm 1} and \eqref{operator set up 1}, we get
\begin{align}\label{inequality of F11}
\nonumber\|F_{1}(A,\gamma)\|_{C([0,t];C)}\leq&\|A_{0}\|_{C}+ b^{-1}\int_{0}^{t}\|A(u)\|_{C}^{p}du\\
				      \leq& \|A_{0}\|_{C}+  b^{-1}L^{p}t.
\end{align}
and by  \eqref{operator set up 2}, we obtain
\begin{align}\label{inequality of F22}
\nonumber \|F_{2}(A,\gamma)\|_{C([0,t];\mathbb{R})}\leq&e^{-\frac{3}{2}t+B_{t}}\gamma_{0}+\gamma_{0}^{-s}\int_{0}^{t}e^{-\frac{3}{2}(t-u)+B_{t}-B_{u}}\|A(u)\|_{C}^{p}du\\
				      \leq& e^{K}\gamma_{0}+e^{\frac{3}{2}s+Ks+2K}L^{p}t.
\end{align}
Setting
$$T_1:= bL^{-p}, ~T_2:= e^{-\frac{3}{2}s-Ks-2K}L^{-p}\;\text{and}\;\hat{T}:=\min\{T_1,T_2\},$$
it follows from \eqref{constant M}, \eqref{inequality of F11} and \eqref{inequality of F22} that
$$\|F(A,\gamma)\|_{C([0,T\wedge \tau_K, C\times (0,\infty)])}\leq 2+\|A_{0}\|_{C}+e^{K}\gamma_{0}\leq L,$$
which implies immediately that $F(A,\gamma) \in \mathbb{D}.$

Next, let us show (2). Indeed, for all $(A_{1},\gamma_{1}), (A_{2},\gamma_{2}) \in \mathbb{D}$,
\begin{align}
\nonumber&\|F_{1}(A_{1},\gamma_{1})-F_{1}(A_{2},\gamma_{2})\|_{C([0,t];C)}\leq \int_{0}^{t}\Bigg\|\frac{A_{1}^{p}(u)}{\gamma_{1}^{q}(u)+b}-\frac{A_{2}^{p}(u)}{\gamma_{2}^{q}(u)+b}\Bigg\|_{C}du\\
\leq&\int_{0}^{t}\frac{\|A_{1}^{p}(u)-A_{2}^{p}(u)\|_{C}}{\gamma_{1}^{q}(u)+b}du
+\int_{0}^{t}\|A_{2}^{p}(u)\|_{C}\Bigg|\frac{1}{\gamma_{1}^{q}(u)+b}-\frac{1}{\gamma_{2}^{q}(u)+b}\Bigg |du.
\end{align}
Now, let us estimate the first term. Considering the convex combination
$$A_{\lambda}(t):=\lambda A_{1}(t)+(1-\lambda)A_{2}(t), ~~\lambda\in [0,1],$$
we have by \eqref{estimate for the norm C} that
\begin{align}\label{inequality of first term 1}
\nonumber\|A_{1}^{p}(u)-A_{2}^{p}(u)\|_{C}&\leq  p\int_{0}^{1}\|A_{\lambda}(u)\|_{C}^{p-1}\|A_{1}(u)-A_{2}(u)\|_{C}\,d\lambda\\
				      &\leq pL^{p-1}\|A_{1}(u)-A_{2}(u)\|_{C};
\end{align}
which implies that
\begin{align}\label{inequality of first term 2}
\int_{0}^{t}\frac{\|A_{1}^{p}(u)-A_{2}^{p}(u)\|_{C}}{\gamma_{1}^{q}(u)+b}du\leq tpb^{-1}L^{p-1}\|A_{1}-A_{2}\|_{C([0,1],C)}.
\end{align}
Similarly, considering the convex combination
$$\gamma_{\lambda}(t):=\lambda \gamma_{1}(t)+(1-\lambda)\gamma_{2}(t), ~~\lambda\in [0,1],$$
we have that
\begin{align}\label{inequality of second term 1}
\nonumber&\int_{0}^{t}\|A_{2}^{p}(u)\|_{C}\Bigg|\frac{1}{\gamma_{1}^{q}(u)+b}-\frac{1}{\gamma_{2}^{q}(u)+b}\Bigg |du\\
\nonumber & \leq pL^{p}b^{-2}\int_{0}^{t}\int_{0}^{1}|\gamma_{\lambda}(u)|^{p-1}|\gamma_{1}(u)-\gamma_{2}(u)|d\lambda du\\
				     & \leq tpL^{2p-1}b^{-2}\|\gamma_{1}-\gamma_{2}\|_{C([0,t],\mathbb{R})}.
\end{align}

Combining \eqref{inequality of first term 2} and \eqref{inequality of second term 1}, we get that
\begin{align}\label{inequality of F1}
\nonumber &\|F_{1}(A_{1},\gamma_{1})-F_{1}(A_{2},\gamma_{2})\|_{C([0,t], C)}\\
&\leq tpL^{p-1}b^{-1}(1+L^{p}b^{-1})\|(A_{1},\gamma_{1})- (A_{2},\gamma_{2})\|_{C([0,t],C\times \mathbb{R})}.
\end{align}
By a similar argument as above, we ascertain that
\begin{align}\label{inequality of F2}
\nonumber &\|F_{2}(A_{1},\gamma_{1})-F_{2}(A_{2},\gamma_{2})\|_{C([0,t], \mathbb{R})}\\
&\leq te^{2K+(\frac{3}{2}+K)s}\gamma_{0}^{-s}L^{p-1}(1+e^{\frac{3}{2}+K}L\gamma_{0}^{-1})\|(A_{1},\gamma_{1})- (A_{2},\gamma_{2})\|_{C([0,t],C\times \mathbb{R})}.
\end{align}
It now follows from \eqref{inequality of F1} and \eqref{inequality of F2} that there exists $\tilde{T}=\tilde{T}(K,L,\gamma_{0})>0$ such that
the inequality \eqref{distance of F} holds. 
The proof is complete.
\end{proof}

\section{Global Existence }
In this section, we shall establish existence and uniqueness of global positive solutions. To prove the global existence and uniqueness result; i.e., Theorem \ref{theorem1}, we assume that $(A(t),\gamma(t))_{0\leq t\leq \tilde{t}}$ is a solution of \eqref{main equation} such that for all $\omega \in E,$
$$A(\omega)\in C([0,\tilde{t}];C(\overline{D}, \mathbb{R})),\quad \gamma(\omega)\in C([0,\tilde{t}]; \mathbb{R}),$$
and then we prove an \emph{a priori} estimate for  $(A(t),\gamma(t))$ almost surely. \\

First, we need the following results.
\begin{proposition}[It\^o's Lemma]
Suppose that $f = f(t, B_t)\in C^2$, i.e., it has continuous partial derivatives up to order two.  
 Then with probability 1, for all $t>0$ and $x=B_{t},$
\begin{align}\label{Ito's Lemma}
df(t,x)=\Big(\frac{\partial f}{\partial t}+\frac{1}{2}\frac{\partial^2 f}{\partial x^2}\Big)dt+\frac{\partial f}{\partial x}dB_{t}.
\end{align}
\end{proposition}

\begin{lemma}\label{Lemma F1}
For the function $\gamma(t)$, we have the following estimates:
\begin{align}\label{estimate H 1}
\gamma(t)\geq \Big(\frac{\eta}{\tau}\Big)^{\frac{1}{s+1}}e^{-\frac{3}{2\eta}t+\frac{1}{\sqrt{\eta}}B_t}\gamma_{0} \quad  t>0,
\end{align}
\begin{align}\label{estimate H 2}
\inf_{0\leq s\leq t}\gamma(s)\geq \Big(\frac{\eta}{\tau}\Big)^{\frac{1}{s+1}}e^{-\frac{3}{2\eta}t-\frac{1}{\sqrt{\eta}}B^*_t}\gamma_{0} \quad  t>0.
\end{align}
\end{lemma}

\begin{proof}
Using It\^o's Lemma and the identity \eqref{operator setting 2}, we have that, for $x=B_{t}$,
\begin{align}\label{computation H0 1}
\nonumber\frac{\partial}{\partial t}\gamma^{s+1}(t)dt=&(s+1)\gamma^{s}(t)\frac{\partial \gamma(t)}{\partial t}dt\\
\nonumber			     =&(s+1)\gamma^{s}(t)\Bigg(-\frac{3}{2\tau}\gamma(t)dt+\frac{1}{\tau}\frac{\overline{A^r}(t)}{\gamma^{s}(t)}dt\Bigg)\\
					    =&-\frac{3}{2\tau}(s+1)\gamma^{s+1}(t)dt+\frac{1}{\tau}(s+1)\overline{A^r}(t)dt,
\end{align}
\begin{align}\label{computation H0 2}
\frac{\partial}{\partial x}\gamma^{s+1}(t)dB_t=(s+1)\gamma^{s}(t)\frac{\partial \gamma(t)}{\partial x}dB_t=\frac{\sqrt{\eta}}{\tau}(s+1)\gamma^{s+1}(t)dB_t,
\end{align}
\begin{align}\label{computation H0 3}
\frac{\partial^{2}}{\partial x^{2}}(\gamma(t))^{s+1}dt=\frac{\partial}{\partial x}\Big(\frac{\sqrt{\eta}}{\tau}(s+1)(\gamma(t))^{s+1}\Big)dt=\frac{1}{\tau}(s+1)^{2}\gamma^{s+1}(t)dt.
\end{align}
It follows from \eqref{computation H0 1} -- \eqref{computation H0 3} that
\begin{align}\label{results 111111}
\nonumber &\tau d\gamma^{s+1}(t)=\tau\Big(\frac{\partial }{\partial t}\gamma^{s+1}(t)+\frac{1}{2}\frac{\partial^2 }{\partial x^2}\gamma^{s+1}(t)\Big)dt+\tau\frac{\partial }{\partial x}\gamma^{s+1}(t)dB_{t}\\
=& \frac{1}{2}(s+1)(s-2)\gamma^{s+1}(t)dt+\sqrt{\eta}(s+1)\gamma^{s+1}(t)dB_t +(s+1)\overline{A^r}(t)dt;
\end{align}
which implies that
\begin{align}\label{results 1111111}
\nonumber \tau\gamma^{s+1}(t)= &\eta e^{-\frac{3}{2\eta}(s+1)t+\frac{1}{\sqrt{\eta}}(s+1)B_t}\gamma_{0}^{1+s}\\
\nonumber &+\eta(s+1)\int_{0}^{\frac{t}{\eta}}e^{-\frac{3}{2\eta}(s+1)(t-u)+\frac{1}{\sqrt{\eta}}(s+1)(B_t-B_u)}\overline{A^r}(u)du\\
\geq&\eta e^{-\frac{3}{2\eta}(s+1)t+\frac{1}{\sqrt{\eta}}(s+1)B_t}\gamma_{0}^{1+s};
\end{align}
from which we derive the estimates \eqref{estimate H 1} and \eqref{estimate H 2}. 
The proof is complete.
\end{proof}

\begin{lemma}\label{Lemma F11}
For every constant $\delta>0$, define the function
\begin{align}\label{definition h}
h_{\delta}(x,t):=\frac{A^{r}(x,t)}{\gamma^{s+1+\delta}(t)},\;\quad(x,t)\in D\times [0,T).
\end{align}
Then, $h_{\delta} \in L^{1}\left(D\times [0,T)\right)$ almost surely, and one has that 
\begin{align}\label{estimate h}
\int_{0}^{\tilde{t}}\int_{D}h_{\delta}(x,t)dxdt\leq \frac{\tau}{\delta\gamma_{0}^{\delta}}+\frac{\delta-3}{2\gamma_{0}^{\delta}}\tilde{t}\Big(\frac{\eta}{\tau}\Big)^{-\frac{\delta}{s+1}}e^{\frac{3\delta}{2\eta}+\frac{\delta}{\eta} K}+\sqrt{\eta}\sup_{0\leq t\leq \tilde{t}}\Bigg|\int_{0}^{t} \frac{1}{\gamma^{\delta}(s)}dB_{s}\Bigg|.
\end{align}
\end{lemma}
\begin{proof}
By using a similar argument as in Lemma \ref{Lemma F1}, we have that
\begin{align}\label{computation H1 1}
\tau\frac{\partial}{\partial t}\gamma^{-\delta}(t)dt
					    =-\frac{3\delta}{2}\gamma^{-\delta}(t)dt-\delta\dfrac{\overline{A^r}(t)}{\gamma^{s+1+\delta}(t)}dt,
\end{align}
\begin{align}\label{computation H1 2}
\tau\frac{\partial}{\partial x}\gamma^{-\delta}(t)dB_t=-\delta\sqrt{\eta}\gamma^{-\delta}(t)dB_t,
\end{align}
\begin{align}\label{computation H1 3}
\tau\frac{\partial^{2}}{\partial x^{2}}\gamma^{-\delta}(t)dt=\delta^{2}\gamma^{-\delta}(t)dt.
\end{align}
It follows from \eqref{computation H1 1} -- \eqref{computation H1 3} that
\begin{align}\label{results 11111}
\tau d\gamma^{-\delta}(t)=\frac{1}{2}\delta(\delta-3)\gamma^{-\delta}(t)dt-\delta\sqrt{\eta}\gamma^{-\delta}(t)dB_t -\delta\dfrac{\overline{A^r}(t)}{\gamma^{s+1+\delta}(t)}dt.
\end{align}
This implies by \eqref{estimate H 2} that
\begin{align*}
\int_{0}^{\tilde{t}}\dfrac{\overline{A^r}(t)}{\gamma^{s+1+\delta}(t)}dt\leq & \frac{\tau}{\delta}\gamma_{0}^{-\delta}+\frac{1}{2}(\delta-3)\int_{0}^{\tilde{t}}\gamma^{-\delta}(t)dt+\sqrt{\eta}\sup_{0\leq t\leq \tilde{t}}\Bigg|\int_{0}^{t} \frac{1}{\gamma^{\delta}(s)}dB_{s}\Bigg|\\
\leq &\frac{\tau}{\delta\gamma_{0}^{\delta}}+\frac{\delta-3}{2}\tilde{t}\Big(\frac{\eta}{\tau}\Big)^{-\frac{\delta}{s+1}}e^{\frac{3\delta}{2}+\frac{\delta}{\eta} K}\gamma_{0}^{-\delta}+\sqrt{\eta}\sup_{0\leq t\leq \tilde{t}}\Bigg|\int_{0}^{t} \frac{1}{\gamma^{\delta}(s)}dB_{s}\Bigg|.
\end{align*}
Now, it suffices to show that
$$\sup_{0\leq t\leq \tilde{t}}\Bigg|\int_{0}^{t} \frac{1}{\gamma^{\delta}(s)}dB_{s}\Bigg|<\infty\quad \mbox{almost surely}.$$
Indeed, using  H\"older's inequality, martingale inequality and It\^o's  isometry, we have  by \eqref{estimate H 2} that
\begin{align*}
&\mathbb{E}\sup_{0\leq t\leq \tilde{t}}\Bigg|\int_{0}^{t} \frac{1}{\gamma^{\delta}(s)}dB_{s}\Bigg|\leq\Bigg(\mathbb{E}\sup_{0\leq t\leq \tilde{t}}\Bigg|\int_{0}^{t} \frac{1}{\gamma^{\delta}(s)}dB_{s}\Bigg|^{2}\Bigg)^{1/2}\\
\leq&\sqrt{2}\Bigg(\mathbb{E}\Bigg|\int_{0}^{t} \frac{1}{\gamma^{\delta}(s)}dB_{s}\Bigg|^{2}\Bigg)^{1/2}
\leq\sqrt{2}\Bigg(\mathbb{E}\int_{0}^{t} \frac{1}{\gamma^{2\delta}(s)}dB_{s}\Bigg)^{1/2}\\
\leq&\sqrt{2}\gamma_{0}^{-\delta}\Big(\frac{\eta}{\tau}\Big)^{-\frac{\delta}{s+1}}\Bigg(\int_{0}^{\tilde{t}}\mathbb{E}e^{\frac{3}{\eta}t-\frac{2}{\sqrt{\eta}}B^*_t}dt\Bigg)^{1/2}<\infty.
\end{align*}
The proof is complete.
\end{proof}

\begin{lemma}\label{Lemma F2}
For any two constants $\alpha> 1,\beta\geq 0$, define the function
\[h_{\alpha,\beta}(t,B_{t}):=\int_{D}\dfrac{A^{\alpha}(x,t)}{\gamma^{\beta}(t,B_{t})}dx, \quad 0\leq t < T.\]
$$\dfrac{p-1}{r}<\min\left\lbrace \frac{2}{N+2}, \frac{q}{s+1}\right\rbrace,$$
it follows that for all $\omega \in E$ defined in \eqref{set E} up to negligible set,
\begin{equation}\label{diff. ineq}
dh{_{\alpha,\beta}}\leq \Big(\dfrac{1}{2\tau}(3\beta+\beta^2)-\alpha\Big)h{_{\alpha,\beta}}dt-\Big(\dfrac{\sqrt{\eta}\beta}{\tau}+\beta\Big)h{_{\alpha,\beta}}dB_{t}+v(t)h_{\alpha,\beta},
\end{equation}
where $v(t)$ is an integrable function on $(0,T)$, almost surely.
\end{lemma}
\begin{proof}
Let $\alpha> 1$ and $\beta\geq 0$. Using It\^o's Lemma, we have that for $s=B_{t},$
\begin{align*}
dh{_{\alpha,\beta}(t,s)}=\Big(\dfrac{\partial h{_{\alpha,\beta}}}{\partial t}+\dfrac{1}{2}\dfrac{\partial^2 h{_{\alpha,\beta}}}{\partial s^2}\Big)dt+\dfrac{\partial h{_{\alpha,\beta}}}{\partial s}dB_{t}.
\end{align*}
By using similar arguments as in \eqref{computation H0 2} -- \eqref{computation H0 3} and \eqref{results 11111}, we get
\begin{align}\label{dh estimate 1}
&\nonumber\dfrac{\partial h{_{\alpha,\beta}} }{\partial t}dt
\nonumber=\int_{D}\left[\alpha\dfrac{A^{\alpha -1}}{\gamma^{\beta}}A_{t}dt-\beta\dfrac{A^{\alpha}}{\gamma^{\beta+1}}\frac{\partial\gamma}{\partial t}dt\right]dx\\
\nonumber=&\int_{D}\left[\alpha\dfrac{A^{\alpha -1}}{\gamma^{\beta}}\left(\epsilon^{2}\Delta A-A+\dfrac{A^{p}}{\gamma^{q}+b}\right)dt-\beta\dfrac{A^{\alpha}}{\gamma^{\beta+1}}\left(-\frac{3}{2\tau}\gamma dt+\frac{1}{\tau}\frac{\overline{A^r}(t)}{\gamma^{s}}dt\right)\right]dx\\
\nonumber=&\left(\dfrac{3\beta}{2\tau}-\alpha \right)h_{\alpha,\beta}dt+\alpha\epsilon^{2}\int_{D}\dfrac{A^{\alpha -1}}{\gamma^{\beta}}\Delta A dx dt+\alpha\int_{D}\dfrac{A^{\alpha +p-1}}{\gamma^{\beta}\left(\gamma^q+b \right)}dxdt\\
&\hspace{0.25truein} -\dfrac{\beta}{\tau}\int_{D}\dfrac{A^{\alpha}\overline{A^r}}{\gamma^{\beta +s+1}}dxdt,
\end{align}
\begin{align}\label{dh estimate 2}
\dfrac{\partial^2 h{_{\alpha,\beta}}}{\partial s^2}dt=\frac{\beta^2}{\tau}\int_{D}\dfrac{A^{\alpha }}{\gamma^{\beta}}dxdt=\frac{\beta^2}{\tau}h{_{\alpha,\beta}},
\end{align}
and
\begin{align}\label{dh estimate 3}
\dfrac{\partial h{_{\alpha,\beta}}}{\partial s}dB_t=-\frac{\sqrt{\eta}\beta}{\tau}\int_{D}\dfrac{A^{\alpha }}{\gamma^{\beta}}dxdB_t=-\frac{\sqrt{\eta}\beta}{\tau} h{_{\alpha,\beta}}dB_t.
\end{align}

Therefore, \eqref{dh estimate 1} -- \eqref{dh estimate 3} imply that
\begin{align*}
dh{_{\alpha,\beta}(t,s)}=&\Big(\dfrac{1}{2\tau}(3\beta+\beta^2)-\alpha\Big)h{_{\alpha,\beta}(t,s)}dt-\frac{\sqrt{\eta}\beta}{\tau} h{_{\alpha,\beta}(t,s)} dB_{t}\\
&+\Bigg(\alpha\epsilon^{2}\int_{D}\dfrac{A^{\alpha -1}}{\gamma^{\beta}}\Delta A dx +\alpha\int_{D}\dfrac{A^{\alpha +p-1}}{\gamma^{\beta}\left(\gamma^q+b \right)}dx-\dfrac{\beta}{\tau}\int_{D}\dfrac{A^{\alpha}\overline{A^r}}{\gamma^{\beta +s+1}}dx \Bigg)dt .
\end{align*}
Since
\begin{align*}
\alpha\epsilon^{2}\int_{D}\dfrac{A^{\alpha -1}}{\gamma^{\beta}}\Delta A dx &=\alpha\epsilon^{2}\int_{\partial D}\dfrac{A^{\alpha -1}}{\gamma^{\beta}}\nabla A\cdot\nu dS-\alpha\epsilon^{2}(\alpha -1)\int_{D}\dfrac{A^{\alpha -2}}{\gamma^{\beta}}\lvert\nabla A\rvert^{2}dx\\
&=\alpha\epsilon^{2}\int_{\partial D}\dfrac{A^{\alpha -1}}{\gamma^{\beta}}\left(\dfrac{-aA}{\epsilon}\right)dS-\alpha\epsilon^{2}(\alpha -1)\int_{D}\dfrac{A^{\alpha -2}}{\gamma^{\beta}}\lvert\nabla A\rvert^{2}dx\\
&=-\alpha\epsilon a\int_{\partial D}\dfrac{A^{\alpha}}{\gamma^{\beta}}dS-\alpha\epsilon^{2}(\alpha -1)\int_{D}\dfrac{A^{\alpha -2}}{\gamma^{\beta}}\lvert\nabla A\rvert^{2}dx\\
&\leq -\alpha\epsilon^{2}(\alpha -1)\int_{D}\dfrac{A^{\alpha -2}}{\gamma^{\beta}}\lvert\nabla A\rvert^{2}dx,
\end{align*}
we obtain the following inequality,
\begin{align}\label{inequality of h alpha beta 11}
 dh{_{\alpha,\beta}(t,s)}\leq&\Big(\dfrac{1}{2\tau}(3\beta+\beta^2)-\alpha\Big)h{_{\alpha,\beta}(t,s)}dt-\frac{\sqrt{\eta}\beta}{\tau} h{_{\alpha,\beta}(t,s)} dB_{t}+E_{1}+E_{2},
\end{align}
where
\begin{align}\label{inequality of h alpha beta 1111}
E_{1} =-\alpha\epsilon^{2}(\alpha -1)\int_{D}\dfrac{A^{\alpha -2}}{\gamma^{\beta}}\lvert\nabla A\rvert^{2}\,dxdt,
\end{align}
and
\begin{align}\label{inequality of h alpha beta 1111}
E_{2}=\alpha\int_{D}\dfrac{A^{\alpha +p-1}}{\gamma^{\beta}\left(\gamma^q+b \right)}\,dxdt.
\end{align}

Now, we concentrate on estimates of $E_{1}$ and $E_{2}$. To do so, let us define the number $0<\kappa<1$ by
\[\kappa=:\dfrac{p-1}{r}=\dfrac{q}{s+1+\delta}~\mbox{for some $\delta>0$}~~\mbox{and}~~ (p-1)<\kappa r,~ q=\kappa (s+1+\delta).\]
Then we obtain
\begin{align}\label{formula of A 11}
\dfrac{A^{\alpha+p-1}}{\gamma^{q}}
&={\left(\dfrac{A^{r}}{\gamma^{s+1+\delta}}\right)^{\kappa}}A^{\alpha}=(h_{\delta})^{\kappa}z^{2},
\end{align}
where $h_{\delta}=\dfrac{A^r}{\gamma^{s+1+\delta}}$ is defined in the statement of Lemma \ref{Lemma F11} and
\begin{align}\label{definition of z and h alpha}
z:=A^{\alpha/2}.
\end{align}
Notice that  \eqref{definition of z and h alpha} implies that
\begin{align}\label{estimate of nabla A}
|\nabla z|^{2}=\frac{\alpha^2}{4}A^{\alpha-2}|\nabla A|^{2}.
\end{align}
By H\"older's inequality, it follows from \eqref{formula of A 11} -- \eqref{definition of z and h alpha} that
\begin{align}\label{inequality of A 11}
\int_{D}\dfrac{A^{\alpha+p-1}}{\gamma^{q}}dx=\int_{D}(h_{\delta})^{\kappa}z^{2}\leq \|h_{\delta}\|^{\kappa}_{L^{1}(D)}\|z\|^{2}_{L^{\frac{2}{1-\kappa}}(D)}.
\end{align}

Since $0<\kappa \leq \frac{2}{N+2}<\frac{2}{N}$, it follows from  Gagliardo-Nirenberg inequality (see e.g. \cite{Pazy}) that there is a constant $C=C(D,N,\kappa)>0$ such that for $\theta:=\frac{N\kappa}{2}\in (0,1),$
\begin{align}\label{inequality of z 11}
\nonumber\|z\|^{2}_{L^{\frac{2}{1-\kappa}}(D)} &\leq C\Big[ \|\nabla z\|^{\theta}_{L^{2}(D)}\|z\|^{1-\theta}_{L^{2}(D)}+\|z\|_{L^{2}(D)}\Big]^{2}\\
&\leq 4C\Big[ \|\nabla z\|^{2\theta}_{L^{2}(D)}\|z\|^{2(1-\theta)}_{L^{2}(D)}+\|z\|^{2}_{L^{2}(D)}\Big].
\end{align}

It follows from \eqref{inequality of A 11} and \eqref{inequality of z 11} that

\begin{align}\label{inequality of A 1111}
\int_{D}\dfrac{A^{\alpha+p-1}}{\gamma^{q}}dx\leq 4C\|h_{\delta}\|^{\kappa}_{L^{1}(D)}\Big[ \|\nabla z\|^{2\theta}_{L^{2}(D)}\|z\|^{2(1-\theta)}_{L^{2}(D)}+\|z\|^{2}_{L^{2}(D)}\Big].
\end{align}

Since by Young's inequality one has that, for $\lambda>0$,
\begin{align}\label{inequality of A 111111}
\|h_{\delta}\|^{\kappa}_{L^{1}(D)} \|\nabla z\|^{2\theta}_{L^{2}(D)}\|z\|^{2(1-\theta)}_{L^{2}(D)}\leq \theta \lambda^{1/\theta} \|\nabla z\|^{2}_{L^{2}(D)} +\frac{1-\theta}{\lambda^{\frac{1}{1-\theta}}}\|h_{\delta}\|^{\frac{\kappa}{1-\theta}}_{L^{1}(D)} \|z\|^{2}_{L^{2}(D)},
\end{align}
then by choosing $\lambda >0$ sufficiently small such that
$$4\alpha \theta C\lambda^{1/\theta} < \frac{\epsilon^{2}(\alpha-1)}{\alpha},$$
one has that, by  using \eqref{estimate of nabla A} and  \eqref{inequality of A 1111} -- \eqref{inequality of A 111111},
\begin{align}\label{inequality of E1 and E2}
\nonumber E_{1}+E_{2} & \leq \gamma^{-\beta}\Big(-\alpha^{-1}\epsilon^{2}(\alpha -1)\|\nabla z\|^{2}_{L^{2}(D)}+4\alpha C\nonumber\|h_{\delta}\|^{\kappa}_{L^{1}(D)} \|\nabla z\|^{2\theta}_{L^{2}(D)}\|z\|^{2(1-\theta)}_{L^{2}(D)}\\
\nonumber &\quad +4\alpha C\|h_{\delta}\|^{\kappa}_{L^{1}(D)}\|z\|^{2}_{L^{2}(D)}\Big)dt\\
\nonumber &\leq \gamma^{-\beta}\Big(-\alpha^{-1}\epsilon^{2}(\alpha -1)\|\nabla z\|^{2}_{L^{2}(D)}+4\alpha \theta C\lambda^{1/\theta} \|\nabla z\|^{2}_{L^{2}(D)}\\
\nonumber &\quad +\Big[\frac{1-\theta}{\lambda^{\frac{1}{1-\theta}}}\|h_{\delta}\|^{\frac{\kappa}{1-\theta}}_{L^{1}(D)} +4\alpha C\|h_{\delta}\|^{\kappa}_{L^{1}(D)}\Big]\|z\|^{2}_{L^{2}(D)}\Big)dt\\
&\leq C_{1}\Big[\|h_{\delta}\|^{\frac{\kappa}{1-\theta}}_{L^{1}(D)} +\|h_{\delta}\|^{\kappa}_{L^{1}(D)}\Big]h{_{\alpha,\beta}}dt,~~C_{1}>0.
\end{align}
Therefore, by using \eqref{inequality of h alpha beta 11},
\begin{align}\label{inequality of nebra z}
dh{_{\alpha,\beta}}\leq \Big(\dfrac{1}{2\tau}(3\beta+\beta^2)-\alpha\Big)h{_{\alpha,\beta}(t,s)}dt-\frac{\sqrt{\kappa}\beta}{\tau} h{_{\alpha,\beta}(t,s)} dB_{t}+C_{1}v(t)h_{\alpha,\beta}dt,
\end{align}
where
$$v(t)=\|h_{\delta}(t)\|^{\frac{\kappa}{1-\theta}}_{L^{1}(D)} +\|h_{\delta}(t)\|^{\kappa}_{L^{1}(D)},$$
Since
$\frac{\kappa}{1-\theta}\leq 1$, it follows from Lemma \ref{Lemma F11} that $v(t)$ is integrable on $(0,T)$ almost surely. The proof is complete.

\end{proof}

\begin{lemma}\label{Lemma F4}
Under the conditions in Lemma \ref{Lemma F2} and the set $E$ defined in \eqref{set E} up to negligible set, there exists a constant $C(T):=C_{\alpha,\beta}(T)\leq\infty$ such that for all $t\in[0,T)$ and  $\omega \in E \subset \Omega$,
\begin{equation}\label{C(T)}
h_{\alpha,\beta}(t,\omega)\leq C(T).
\end{equation}
\end{lemma}
\begin{proof}
Using Ito's Lemma and inequality \eqref{diff. ineq} in Lemma \ref{Lemma F2}, it follows that for all $t\in[0,T)$ and  $\omega \in E \subset \Omega$,

\begin{align}\label{final estimate}
d\left[ \exp \Bigg(-\Big(\dfrac{1}{2\tau}(3\beta+\beta^2)-\alpha\Big)t+C_{1}\int_{0}^{t}v(s)ds+\dfrac{\sqrt{\eta}\beta}{\tau}B_{t}(\omega)\Bigg)h_{\alpha,\beta}(t,\omega)\right] &\leq 0.
\end{align}

Integrating from 0 to $t$, we get from \eqref{final estimate} that
\begin{align*}
h_{\alpha,\beta}(t,\omega) &\leq C_{2} \exp \Bigg(-C_{1}\int_{0}^{T}v(s)ds\Bigg)h_{\alpha,\beta}(0,\omega),
\end{align*}
where
$$C_{2}=\exp \Bigg(\Big(\dfrac{1}{2\tau}(3\beta+\beta^2)-\alpha\Big)T \Bigg)\exp \Bigg(\dfrac{\sqrt{\eta}\beta}{\tau}\sup_{0\leq t\leq T}|B_{t}| \Bigg).$$
The proof is complete.\\
\end{proof}

From Lemma \ref{Lemma F1} and Lemma \ref{Lemma F4}, we deduce the Corollary below.
\begin{cor}\label{cor}
Let $\ell\geq 1$ and all other assumptions in Theorem \ref{theorem1}, Lemma \ref{Lemma F2} and Lemma \ref{Lemma F4} hold true. Define
\[g_{1}(A,\gamma)=\dfrac{A^{p}}{\gamma^{q}+b},\]
\[g_{2}(A,\gamma)=\dfrac{A^{r}}{\gamma^{s}},\]
then there exist positive constant $C_{\ell}(T)$, such that
\[\left\Vert g_{j}(A,\gamma)\right\Vert_{L^{\ell}(\Omega)}\leq C_{\ell}(T) \qquad j=1,2 \]
for all $0\leq t<T$.
\begin{proof}
The proof to this Corollary follows from  Lemma \ref{Lemma F4}.
\end{proof}
\end{cor}

\begin{proof}[Proof of Theorem 1.1]
Under the conditions in Lemma \ref{Lemma F2} and the set $E$ defined in \eqref{set E} up to negligible set, using \eqref{estimate for the norm C}, \eqref{operator setting 1} and Corollary \ref{cor}, we have that for all $0\leq  t\leq T$,
\begin{align}\label{operator setting 11}
\nonumber \|A(t)\|_{L^{2}(D)}\leq&\|S(t)A_{0}\|_{L^{2}(D)}+\int_{0}^{t}\Big\|S(t-u)\Big(\frac{A^{p}(u)}{\gamma^{q}(u)+b}\Big)\Big\|_{L^{2}(D)}du\\
\nonumber \leq& \|A_{0}\|_{L^{2}(D)}+T\Big\|\frac{A^{p}(u)}{\gamma^{q}(u)+b}\Big\|_{L^{2}(D)} \\
\leq& \|A_{0}\|_{L^{2}(D)}+TC_2(T)
\end{align}
In addition, one is able to obtain the estimate \eqref{estimate for gamma} from  \eqref{estimate H 2}. With these estimates, the unique local solution obtained in Proposition 2.1 may now be continued indefinitely to obtain a global solution. The proof is complete.
\end{proof}

\end{document}